\documentclass[12pt]{article}
\newcommand{\comment}[1]{}

\newcommand{\reals}{{\bf R}}

\newcommand{\BibTeX}{{\rm B\kern-.05em{\sc i\kern-.025em b}\kern-.08em     
    T\kern-.1667em\lower.7ex\hbox{E}\kern-.125emX}}

\usepackage{amsmath,amsthm,amscd,amssymb}
\usepackage{hyperref}
\newcommand{\spn}{\mathop{\rm span} \nolimits}

\usepackage{chngcntr}
\usepackage{apptools}
\AtAppendix{\counterwithin{equation}{section}}

\renewcommand{\reals}{\mathbf{R}}

\newenvironment{example}[1][Example]{\begin{trivlist}\item[\hskip \labelsep
{\it #1. }]}{  \goodbreak \end{trivlist}}   
\numberwithin{equation}{section} 
\newtheorem{theorem}[equation]{Theorem}

\newtheorem{lemma}[equation]{Lemma}

\begin{document}

\title{Estimates for more Brascamp-Lieb forms in $L^p$-spaces with
  power weights}

\author{
 R.M.~Brown\footnote{
Russell Brown   is  partially supported by a grant from the Simons
Foundation (\#422756).
} \\ Department of Mathematics\\ University of Kentucky \\
Lexington, KY 40506-0027, USA
\and
K.A.~Ott\footnote{Katharine Ott is partially supported by a grant from the Simons
Foundation (\#526904).} 
\\Department of Mathematics\\Bates College\\Lewiston, ME  04240-6048, USA
}

\date{}

\maketitle

\begin{abstract}
We consider a class of Brascamp-Lieb forms and give conditions which guarantee the 
boundedness of these form on $L^p$-spaces with weights that are a power of the 
distance to the origin. These conditions are close to necessary and sufficient. 
\end{abstract}

\section{Introduction} \label{Introduction}
The goal of this paper is to give conditions which guarantee the 
boundedness of certain Brascamp-Lieb forms on weighted $L^p$-spaces.   
To describe the forms we  study, we let 
$$
E= \{ v_1, \dots, v_N\} \subset \reals ^m \setminus \{0\}$$
be a finite collection of non-zero vectors which does not contain 
a  pair of 
collinear vectors.   We fix $ k \geq 1$ and for 
$ v= ( v^1, \dots, v^m)  \in \reals ^m$ and $ x = (x^1, \dots, x^m ) \in \reals ^ { mk}$, 
we define $ v\cdot x \in \reals ^k$ by 
$ v\cdot x = \sum _ { i =1 }^m v^i x^i.$ If $ f_1 \dots, f_N$ are non-negative, measurable 
functions on $ \reals ^k$, the form we will study is 
\begin{align}\label{formdef}
\Lambda(f_1,\dots,f_N)=\int_{\reals^{km}} \prod_{i=1}^N f_i(v_i\cdot x)\, dx.
\end{align}

We will use $L^p _ \alpha ( \reals ^k)$ to denote the weighted $L^p$-space with 
the norm
$$
\|f\|_{ L_\alpha^p ( \reals ^k )}= \| |\cdot | ^ \alpha f\|_{ L^ p ( \reals ^k)}.
$$
Going forward, all norms will be over $\reals^k$.
Our goal, which is only partially achieved, is to characterize the set of 
indices $(1/p_j,\lambda _j ) _{ j=1} ^N$ for which we have the estimate
\begin{align}\label{mainest}
\Lambda(f_1,\dots,f_N) \leq C\prod_{j=1}^N\|f_j\|_{L^{p_j}_{\lambda_j}}.
\end{align}
The constant $C$ may depend on the set of vectors $E$, the dimension $k$, 
and the indices $ ( 1/p_j, \lambda _j)_{ j=1} ^N$. 
We will focus most of our attention on the case where $ E= \{ v_1, \dots v_N\} 
\subset \reals ^m$ has the property that any subset $K \subset E$ with cardinality 
$\#K =m$ is a basis for $\reals ^m$. We call such a set $E$ {\em generic}. 

We will establish the following theorem, which gives conditions on the indices 
$ ( 1/ p _ j, \lambda _j )_{j=1} ^N $ that guarantee that the bound \eqref{mainest} holds. 

\begin{theorem} \label{suffthm}
Let $ E = \{v_1, \dots, v_N\}$ be a generic set in $\reals^m$. 
Suppose that $(1/p_j, \lambda_j )_{j=1}^N\in ((0,1)\times \reals)^N$ 
and that the following list of conditions are true:
\begin{align}
& \sum_{j=1}^N(\frac{1}{p_j}+\frac  {\lambda_j} k ) = m,\label{Scaling}\\
& \sum_{v_j\notin V } (\frac{1}{p_j}+ \frac {\lambda_j} k)  > m-\dim (V) \quad
\mbox{for all non-zero, proper subspaces $V$ of   $\reals^m$},
%%& \sum_{\{j:v_j\cdot V\neq \{0\}\}} (\frac{1}{p_j}+ \frac {\lambda_j} k)  > \dim (V) \quad
%%\mbox{for all proper subspaces }\, V \subset\reals^m,
\label{v1}\\
& \sum _ { v_j \notin V } \lambda _j \geq 0\quad 
\mbox{for $V$ a subspace of $\reals ^m$}, \label{v2} \\
%%& \sum_{\{j:v_j\cdot V\neq \{0\}\}}\lambda_j\geq 0 \quad
%%\mbox{for all subspaces }\, V \subset\reals^m,\label{v2}\\
& \sum_{j=1}^N\frac{1}{p_j}\geq 1.\label{I}
\end{align}
Then estimate \eqref{mainest} holds.
\end{theorem}

%%Corollary here. With \sum 1/r_j \geq 1, we have an estimate using $L^{ p_j, r_j} _\lambda _j$
The next result gives necessary conditions which are close to the 
sufficient conditions in the previous theorem. The set of indices 
$(1/p_j, \lambda _j)_{ j=1}^N$ which satisfy our necessary conditions in
 Theorem \ref{necthm} below form a polytope in the hyperplane defined by 
  equation \eqref{Scaling}. The above result, Theorem \ref{suffthm}, 
 tells us that the estimate  \eqref{mainest} holds on the interior of
 this polytope.  The estimate may hold on part of the boundary,
 but we do not investigate results on the boundary here. 
\begin{theorem}\label{necthm}
%%Let $ = \{v_1, \dots, v_N\}$ be a generic set in $\reals^m$. 
Let $E=\{v_1, \dots, v_N\} \subset \reals ^m$ be a set of vectors and assume that
no pair of vectors from $E$ are collinear. 
If  \eqref{mainest} holds,  then the indices $(1/p_j, \lambda_j)_{j=1} ^N
\in ( [0,1]  \times\reals)^N$ satisfy \eqref{Scaling},
\eqref{I},   \eqref{v2}, and 
\begin{align}\label{v1c}
\sum_{ v_j \notin V}(\frac{1}{p_j}+\frac {\lambda_j}k )\geq m -\dim(V) \quad
\mbox{for all subspaces }\, V \subset\reals^m.
\end{align}
\end{theorem}

There is a great deal of work on Brascamp-Lieb forms dating back 
to the original work of 
Brascamp, Lieb, and Luttinger \cite{MR0346109}. 
Two papers that give recent developments and include more 
of the history are by Bennett, Carbery, Christ and Tao 
\cite{MR2661170} and Carlen, Lieb and Loss \cite{MR2077162}. 

A good place to start our story is a paper of 
Barthe \cite[Proposition 3]{MR1650312} who gives a necessary and sufficient 
condition on the indices $ (1/p_1, \dots, 1/p_N)$ for which 
we have estimate \eqref{mainest} in the case when exponents for the weights, 
$ \lambda _j$, are all zero. In fact, he shows that the 
set of indices $(1/p_1, \dots, 1/p_N)$  is the  
matroid polytope for the set of vectors $E=\{v_1, \dots, v_N\}$ 
(though he does not use this terminology). His condition was generalized by 
Bennett {\em et al. }and we make fundamental 
use of these results in our work. Barthe's work is not restricted 
to forms constructed using 
generic sets of vectors and it is an interesting problem to 
characterize the set of weighted spaces for which we have 
\eqref{mainest} without the assumption that the set of vectors $E$ is a generic set. 

A key point in our proof of Theorem \ref{suffthm} is an extension of results  for 
Brascamp-Lieb forms from Lebesgue spaces to Lorentz spaces by a real interpolation 
argument. This idea dates back 
at least to O'Neil \cite{MR0146673} who studies Young's inequality in Lorentz 
spaces and gives an application to fractional integration. 
It reappears in several places including Christ \cite{MR766216}. 
Our work will rely on results of 
 Bez, Lee, Nakamura and Sawano \cite{MR3665018} who study Brascamp-Lieb forms on Lorentz spaces. 

Our interest in Brascamp-Lieb forms arose from the study of a 
scattering map for a first order system in the plane. This map originated in work of Beals and Coifman \cite{BC:1988} and Fokas and Ablowitz \cite{FA:1984} 
who observed that the scattering map  transforms solutions
of the Davey-Stewartson equations to solutions of a linear evolution equation. In the work of the 
first author \cite{MR1871279}, we expand the  scattering map in a 
series of multi-linear expressions and  use estimates for 
certain Brascamp-Lieb forms to establish that the series 
is convergent in a neighborhood of zero in $L^2$. In 
his dissertation \cite{ZN:2009,MR2754810}, Z. Nie 
gave a proof of this result that relied on 
multi-linear interpolation. 
The works of Perry (with an appendix by M.~Christ) \cite{MR3551174} and Brown, Ott and Perry 
(with an appendix by N. Serpico)  \cite{MR3494243} 
give estimates for the map on weighted Sobolev spaces. The recent work  of 
Nachman, Regev, and Tataru \cite{MR4081134} establishes that the map is 
continuous on all of $L^2$, but does not  rely on our approach using Brascamp-Lieb forms. 

Motivated by this body of work and especially the results of Brown, Ott, and Perry,  
we became interested in the question of characterizing the complete set of 
$L^p$-spaces with power weights for which we have the bound \eqref{mainest}.  A first step in 
this direction appears in our work with Lee \cite{MR4198080} where we consider 
forms constructed using generic sets of vectors in $ \reals ^2$. Our results in \cite{MR4198080} are close to optimal in the sense that we 
find a closed polytope  containing the  set of 
indices $(1/p_j, \lambda _j)_{j=1}^N$ for which 
we have \eqref{mainest}, 
and we show that in the interior
of this polytope we have estimate \eqref{mainest}. 
In the current paper, we extend these methods from $ \reals ^2$ to handle forms based 
on generic sets of vectors in $ \reals ^m$ with $m > 2$. 
The polytope we find  lies in a 
hyperplane in $ \reals^{2N}$ and by interior we mean the interior in the relative 
topology for this hyperplane. 

The methods developed over the arc of our works can be used to establish the estimate \eqref{mainest} for certain indices when the set of vectors do not satisfy the generic condition. In fact, going back to the results of Brown, Ott, and Perry \cite{MR3494243}, the form studied there is defined using non-generic sets of vectors and we were successful in finding estimates for the form. However, when considering arbitrary sets of vectors $E\subset \reals^m$, our methods are only successful if we have the additional condition that $E$ is a generic set.  
To be more precise about this limitation, in the last section of this paper we give an example of a form where 
there are points in the interior of the polytope defined by Theorem 
\ref{suffthm} for which a straightforward 
extension of  our argument used in the 
generic case fails to establish \eqref{mainest}. 

Using duality, the estimates for Brascamp-Lieb forms (at least for the range of indices
$p$ that we consider) are equivalent to estimates for multi-linear
operators.  A good place to start
this part of the story is a theorem of Stein and Weiss \cite{MR0098285}, which gives 
optimal results for the action of the Riesz potential  of order $ \beta$ on spaces $L^ p _ \alpha$. 
Using duality, we can see that estimates for this operator are  equivalent to 
weighted estimates for the form 
$$
\int _ { \reals ^ { 2k} } \frac { f(x) g(y)}{|x-y |^ { k- \beta}} \, dx dy. 
$$
In our earlier work \cite{MR4198080} we showed how to obtain the Stein-Weiss result 
from the case $m=2$ of our Theorem on multilinear forms defined using 
generic sets of vectors in $ \reals ^2$.  Work of Grafakos \cite{MR1164632} 
considers an $n$-linear fractional integration operator 
in the unweighted case. Grafakos's operator may be related to a form 
constructed using a set of generic vectors in $ \reals ^2$. 
Grafakos and Kalton \cite[Section 5]{MR1812822}
consider a multi-linear form involving 
a set of vectors that is not generic. 

Kenig and Stein \cite{MR1682725} consider a multilinear fractional integral 
and give results when their operator maps into an $L^p$-space with $p<1$. This 
is a result that cannot be obtained by using duality and appealing to estimates 
for forms.  There are a number of authors who have considered weighted estimates 
for bilinear fractional integrals. The author Komori-Furuya \cite{MR4047697} has 
studied these operators on the spaces $L^p_\alpha$ we consider here. Moen 
\cite{MR3130311} gives conditions on general weights which guarantee boundedness 
of a bilinear fractional integral.  In comparison, our work considers less general 
weights, but gives results that are close to optimal.  An interesting open 
problem is to find conditions on weights that allow us to establish the 
finiteness of Brascamp-Lieb forms for more general classes of weighted 
spaces. 

The outline of this paper is as follows. In section \ref{SectionSuff} we prove
the conditions which guarantee the finiteness of the form  and \eqref{mainest}. Section \ref{SectionNec} 
is devoted to the proofs of the necessary conditions and section \ref{examples} gives several examples 
illustrating our results and the limitations of our methods. 

We thank our colleague, C.W.~Lee for suggesting that we study Brascamp-Lieb
forms constructed using generic sets of vectors.

\section{Sufficient conditions}\label{SectionSuff}
%% Do we want a proof environment here?
%%\begin{proof}[Proof of Theorem \ref{suffthm}] 
In this section, we give the proof that the conditions in Theorem \ref{suffthm} are sufficient to 
imply the estimate \eqref{mainest}.  
The proof will proceed in four steps of
increasing complexity. First, we prove the theorem in the case when $\lambda_j=0$
for all $j$ and $k=1$. Next, we remove the restriction on $k$.   The third step is to  prove 
the result under the assumption that all $\lambda_j\geq 0$. The 
final step is to show that when  $\lambda_j<0$ for at least one index  
$j$  we may reduce to the case where all $ \lambda _j \geq 0$.

Before beginning the proof, we recall a few facts about the Lorentz spaces, $L^{p,r}$, where $1\leq p < \infty$ and $1\leq r \leq \infty$. A definition of
these spaces  may be 
found in Stein and Weiss \cite[p.~191]{SW:1971}. This family of spaces include the familiar $L^p$-spaces
as $ L^p = L^ { p,p}$ and arise as real interpolation spaces of the $L^p$-spaces.  
An observation of Calder\'on is that  the spaces $L^{p,r}$
may be normed if $1 <p < \infty$ and $ 1\leq r\leq \infty$ 
(and of course when $p=r=\infty$,  $L^{\infty, \infty} = L^\infty$). 
Imitating the definition of $L^p_\alpha$, we  define weighted Lorentz spaces $L^{p,r}_\alpha$ as the collection of functions $f$ for which  $|\cdot |^ \alpha f$ lies in $L^{p,r}$. 

 We  recall the extension of H\"older's inequality to Lorentz 
 spaces due to O'Neil \cite[Theorem 3.4]{MR0146673}. Suppose $p \in (1,\infty)$, 
 $p_1, p_2 \in (1, \infty)$, $ 1/p = 1/p_1 + 1/p_2$, $r, r_1, r_2 \in [1,\infty]$,
and $1/r\leq 1/r_1 + 1/r_2 $. Then there is a finite constant $C$ such that 
\begin{equation} \label{LorentzHolder}
\| f_1f_2 \|_{L^{p,r}} \leq C\| f_1 \|_{ L^ { p_1,r_1}} \|f_2 \| _{ L^{ p_2, r_2}}.
\end{equation}
 %% Are Lorentz spaces L^{\infty, p } defined? Only defined in Stein and Weiss when p=r=\infty. 
One useful observation  is that  for $ 0 < \lambda < k$, 
 the function $|x|^ { -\lambda }$ lies in the space $L^{ \lambda /k, \infty}( \reals ^k)$ 
 and in fact if $ 0 < \lambda - \alpha <k$ we have  
 \begin{equation} \label{LorPower}
 |x|^ { -\lambda} \in L^ { k/(\lambda - \alpha), \infty}_\alpha( \reals ^k).
 \end{equation}

We begin with the case when  $k=1$ and  $\lambda_j= 0$ for all
$j$. Under this restriction, provided  that \eqref{Scaling} and \eqref{v1} hold, 
the theorem of Bennett, Carbery, Christ, and Tao \cite[Theorem 2.1]{MR2661170}
implies that \eqref{mainest} holds.  Next, Proposition 2.1 from our previous work
\cite{MR4198080} shows that the boundedness of the form for $k >1$ follows 
from the case where $k=1$. From here, a result of Bez {\em et~al.~}\cite[Theorem 1]{MR3665018} allows us
to conclude that we have an estimate for the form in Lorentz spaces. In particular, if the indices 
$(1/p_j, 0) _{j=1}^N \in (0,1)^N \times \{0\}^N$ satisfy
\eqref{Scaling}, \eqref{v1} and the indices $ (1/r_1, \dots, 1/r_N)$ satisfy 
$ \sum _ { j=1} ^N 1/r_j\geq 1$, then  we have
\begin{equation}
\label{Bez}
\Lambda(f_1, \dots, f_N) \leq C 
\prod_{ j=1} ^N \|f_j \| _{ L^{ p_j, r_j}}. 
\end{equation}
The result of Bez {\em et~al.~}depends on a multi-linear interpolation 
argument which may be found in work  of M.~Christ \cite{MR766216} or S.~Janson \cite{SJ:1986}. 

Next, suppose that $\lambda_j\geq 0$ for all $j$. Set
$$\frac{1}{r_j} = \frac{1}{p_j}+\frac {\lambda_j} k,$$
and observe that for the one-dimensional space  $\reals  v_j $ we have
$$\big\{v_i: v_i \notin \reals v_j \big\}=\big\{v_1,\dots,v_N\big\}\setminus\{v_j\}.
$$
In this case, the conditions \eqref{Scaling} and \eqref{v1} imply
\begin{equation}  \label{NotOne}
 \frac{1}{p_j}+ \frac {\lambda_j} k = m-\sum_{i\neq j} (\frac{1}{p_i}+\frac {\lambda_i} k )<
m-(m-1)=1.
\end{equation}
Also from the assumptions that $1/p_j \in (0,1)$ and $\lambda_j\geq 0$ for 
all $j$, it immediately follows that $\frac{1}{p_j}+\frac {\lambda_j} k>0$. Finally, employing 
\eqref{Bez}, 
%%the interpolation theorem of Grafakos and Kalton \cite[Theorem 4.7]{MR1812822}, 
the 
generalized H\"older inequality \eqref{LorentzHolder} and the observation 
\eqref{LorPower} (with $\alpha =0$) 
we have
\begin{align*}
\Lambda(f_1, \dots,f_N) &\leq C\prod_{j=1}^N \|f_j\|_{L^{r_j,p_j}}\\
&\leq C\prod_{j=1}^N \|f_j|\cdot|^{\lambda_j}\|_{L^{p_j,p_j}}
\| |\cdot|^{-\lambda_j} \|_{L^{k/\lambda_j,\infty}}  = C\prod_{j=1}^N \|f_j\|_{L^{p_j}_{\lambda_j}}
\end{align*}
which gives \eqref{mainest} in the case that all $ \lambda _j \geq 0$. 
\comment{Need to verify that we are in the interior of the set from BCCT.}

The next step is to prove estimate \eqref{mainest} when $\lambda_j<0$ for some $j$.
We will proceed by induction on the number of indices $j$ for which $ \lambda _j <0 $. The 
following technical Lemma gives most of the details of the induction step. 

%%This lemma reduces the number of  negative lambda's. 
\begin{lemma} \label{InductionStep}
Given a family $({1}/{p_j},\lambda_j)_{j=1} ^N$ 
satisfying (\ref{Scaling}--\ref{I}), 
and $1/p_j\in(0,1)$, with at least one $j$ with 
$\lambda_{j} <0$, let $j_0 $ be an index so that 
$ \lambda _{ j _0 } = \min \{ \lambda _j : j=1, \dots , N\}$. We may find a finite family
$ \{\beta^{(\alpha)}\}_{\alpha \in \zeta}$ 
so that $({1}/{p_j},\beta_j^{(\alpha)})_{j=1}^N$ satisfy (\ref{Scaling}--\ref{I}).
The vectors of exponents $\beta^{(\alpha)}$ are indexed by a family of 
multi-indices $\zeta \subset \{0, 1, \dots, N\}^{\ell-m+1}$ where $\ell$ is the 
number of positive entries in $\lambda$. 
Moreover, $({1}/{p_j},\beta_j^{(\alpha)})$
satisfy
\begin{align}
& \beta_{j_0}^{(\alpha)}=0 \quad \mbox{for all }\, \alpha\in \zeta,\\
& \beta_j^{(\alpha)} = \lambda_j\quad \mbox{if }\, 
\lambda_j < 0\mbox{ and }j \neq j_0,\\
& 0 \leq \beta^{(\alpha)}_j\leq \lambda_j \quad \mbox{if }\, \lambda_j \geq 0.
\end{align}
Also, we have 
\begin{align}\label{estimate2}
\Lambda(f_1,\dots,f_N) \leq C \sum_{\alpha\in \zeta}
\Lambda(|\cdot|^{\lambda_1-\beta_1^{(\alpha)}} f_1, \dots, |\cdot|^{\lambda_N-\beta_N^{(\alpha)}}f_N).
\end{align}
\end{lemma}

If we grant the Lemma, then the proof by induction is quite easy. 
We have already established the  base 
 case  when all the exponents $ \lambda _j $ are non-negative. 
To proceed by induction, we assume that we have Theorem \ref{suffthm} 
when $J$ of the exponents $ \lambda_j$ are negative. If $ \lambda$ is 
a vector of exponents with $ J+1$ 
negative entries and so that $ (1/p_j, \lambda _j ) _ { j=1}^N$ satisfy the conditions of Theorem \ref{suffthm}, then 
using the family $\{ \beta^{(\alpha )} : \alpha \in \zeta \}$ from 
Lemma \ref{InductionStep}, we have
\begin{align*}
\Lambda ( f_1, \dots, f_N) &\leq \sum _ { \alpha \in \zeta } 
\Lambda(|\cdot|^{\lambda_1-\beta_1^{(\alpha)}}f_1,\dots,|\cdot|^{\lambda_N-\beta_N^{(\alpha)}}f_N) \\
& \leq C \sum _ { \alpha \in \zeta } 
\prod_{j=1}^N \||\cdot|^{\lambda_j-\beta_j^{(\alpha)}}f_j
 \|_{L^{p_j}_{\beta_j^{(\alpha)}}} \nonumber 
  = C\prod_{j=1}^N\|f_j\|_{L^{p_j}_{\lambda_j}}.
\end{align*}
Thus, establishing  Lemma \ref{InductionStep}  will  complete the proof of Theorem \ref{suffthm}. 
To prove Lemma \ref{InductionStep}, we will make use of the following lemmata which 
give alternative characterizations of \eqref{v1} and \eqref{v2} in the case of generic sets of vectors. 

\begin{lemma}\label{detour}
Let $E= \{v_1, \dots,v_N\}$ be a generic set in $\reals^m$ and 
%%%%
%%If \eqref{v2} holds, then for each $K\subset\{v_1,\dots,v_N\}$, with $\#K\leq m-1$, 
%%we have 
%
%%$$
%%\sum_{j\notin K} \lambda_j\geq 0.
%%$$
%
let $(1/p_j, \lambda_j ) _{ j=1}^N \in (0,1)^N \times \reals ^N$. 
The following conditions on these indices are equivalent:
\begin{align}
&\sum_{v_{j}\notin W }\lambda_j \geq 0 \quad
&&\mbox{for all proper  subspaces $W\subset  \reals ^m $},\label{detourB} \\
&\sum_{v_j\notin K} \lambda_j \geq 0 \quad 
&&\mbox{for all } K\subset E
%%\{v_1,\dots,v_N\} 
\,\mbox{such that }\, \#K \leq m-1.\label{detourA}
\end{align}
\end{lemma}

\begin{proof}
First we show that \eqref{detourA} implies \eqref{detourB}. Given $W$ a proper 
subspace of $\reals^m$, set $K =\{v_\ell: v_{\ell}\in W \}$. Then
$\#K\leq \dim W \leq m-1$ and from \eqref{detourA} we obtain \eqref{detourB}
\[
\sum_{ v_{i} \notin W} \lambda_i =\sum_{v_i\notin K} \lambda_i \geq 0.
\]

For the other implication, assume that \eqref{detourB} holds. Given $K$, set 
$
W = \spn (K) 
$
which will be a proper subspace of $\reals ^m$ since $ \# K \leq m-1$.
Thanks to our assumption that the set $E= \{ v_1, \dots ,v_N\} $ is generic 
and that $\# K \leq m-1$, it follows that $ W \cap E = K$.
Then it follows that 
$$
\sum_{v_i\notin K} \lambda_i = \sum_{ v_i\notin W } \lambda_i\geq 0.
$$
\end{proof}

In our next Lemma, we need to avoid the cases when $K$ is empty or all of $E$ and 
we will have equality
in \eqref{conditionD}. 
\begin{lemma}\label{detour2}
Let $E= \{v_1,\dots,v_N\}$ be a generic set in $\reals^m$ and suppose 
$ ( 1/p_j, \lambda _j ) _ { j =1}^N \in (0,1)^N \times \reals ^N$
%%Suppose that $\lambda_j <0$ for one $j\in \{1,\dots,N\}$. 
Then the following two statements are equivalent:
\begin{align}
& \sum_{v_j \notin V } (\frac{1}{p_j}+ \frac {\lambda_j}k  )> m- \dim(V) 
\quad\mbox{for all  subspaces  $V$ with $1 \leq  \dim(V) \leq m-1 $}
\label{conditionC} \\
& \sum_{ v_j\notin K} (\frac{1}{p_j} + \frac {\lambda_j} k ) > m-\#K \quad 
\mbox{for all }\, K\subset
E
%% \{v_1, \dots, v_N\}
\,\mbox{with } 1 \leq  \#K\leq m-1. \label{conditionD}
\end{align}
\end{lemma}

\begin{proof}
First, to prove that \eqref{conditionC} implies \eqref{conditionD}, let 
%%  $K\subset \{v_1,\dots,v_N\}$ 
$K \subset E$
with $1 \leq \# K \leq m-1$.  Let $V = \spn (K)$ and use the assumption that $E$
%% $\{v_1,\dots, v_N\}$ 
is a generic set to conclude  $\dim(V)=\#K$  and 
$\{v_j: v_j\notin V \} = E \setminus K$
%% $\{v_1,\dots, v_N\}\setminus K$. 
Thus
\begin{align*}
& \sum_{v_j\notin K} ( \frac{1}{p_j}+\frac {\lambda_j} k ) = \sum_{v_j\notin V }
(\frac{1}{p_j}+\frac {\lambda_j} k )>m- \dim(V) = m-\#K.
\end{align*}

Now suppose that we have \eqref{conditionD} and  
let $V\subset\reals^m$ be a proper subspace. Set $K = \{v_j:
v_j\in V \}$. By the generic condition, 
$
 \#K \leq \dim(V) 
$
which implies  that $m-\dim( V)\leq m-\#K$.
Thus in the case that $ \# K \geq 1$, \eqref{conditionC} follows from 
\eqref{conditionD}.   If $ \# K =0$ and $ \dim V \geq 1$, we may use 
\eqref{conditionD} for the set $ K = \{ v_1\}$ and find that there is at 
least one $j$ for which $1/p _{j_0} + \lambda _{j_0} /k \geq 0$.
Then 
$$\sum _ { j=1}^N (\frac 1 { p_j } + \frac { \lambda _j } k)
\geq \sum _ { j\neq j_0} (\frac 1 { p_j } + \frac { \lambda _j } k)
> m-1 \geq m -\dim V.
$$
Thus \eqref{conditionC} follows from \eqref{conditionD}.
\end{proof}

\begin{proof}[Proof of Lemma \ref{InductionStep}]
For convenience, we assume that the vectors $v_j$ are labeled so that the exponents $\lambda $ are 
in decreasing order 
and $\ell$ is the last index for which $ \lambda _\ell >0 $: 
$$ \lambda_1\geq \dots \geq \lambda_\ell >0 
\geq \lambda_{\ell+1} \geq \dots \geq \lambda_N.$$
With this re-indexing our goal is find a 
family $\{ \beta ^ { (\alpha )} \}$ with $\beta^{ ( \alpha )}_N =0$. 
To begin, we use \eqref{detourA} with $K= \{ v_1, \dots, v_{m-1}\}$ to obtain that 
\begin{align}
\lambda_N+\sum_{i=m}^\ell\lambda_{i} \geq \sum_{i=m}^N\lambda_i\geq 0.
\end{align}
Thus it follows that
\begin{align}\label{step}
|\lambda_N|\leq \sum_{i=m}^\ell \lambda_i.
\end{align}
Now we will proceed by giving a construction involving at most $\ell-m+1$ steps to 
produce the family $\{\beta^{(\alpha)}\}\subset\reals^N$. The vectors 
$\beta^{(\alpha)}$ will be indexed by multi-indices $\alpha\in\{0,1,\dots,N\}^{\ell-m+1}$.

To begin, we set $\beta^{(0)}=\lambda$ and then set
$$
\gamma_0=\min(\lambda_1,\dots,\lambda_m,|\lambda_N|)=\min(\lambda_m, |\lambda_N|).
$$
Now define
\begin{align*}
\beta^{(j_1,0)}& =(\lambda_1,\dots,\lambda_{j_1}-\gamma_0,\dots,\lambda_{m+1},\dots,
\lambda_N+\gamma_0) \nonumber \\
& = \lambda -\gamma_0e_{j_1}+\gamma_0 e_N,
\end{align*}
and set $\zeta_1=\{\beta^{(j_1,0)} : j_1 =1,\dots,m\}.$ 
It is clear that for each $\alpha\in \zeta_1$,
$({1}/{p_j},\beta_j^{(\alpha)})_{j=1}^N$ satisfies \eqref{Scaling}
and 
\eqref{I}. It  remains to show that \eqref{v1}
and \eqref{v2} hold.

We will use Lemma \ref{detour} to show 
that $\beta^{(j_1,0)}$ satisfies \eqref{v2}. Thus let $K \subset E$ be a set
with $\#K \leq m-1$ and  consider the 
four cases according to whether or not $v_{j_1}$ and $v_N$ lie in $K$.
If both $v_{j_1}, v_N$ lie in $K$, or both lie in $E\setminus K$,
then 
\begin{align*}
\sum_{v_j\notin K}\beta_j^{(j_1,0)}=\sum_{v_j\notin K}\lambda_j \geq 0.
\end{align*}
If $v_N\notin K$ and $v_{j_1}\in K$, then
\begin{align*}
\sum_{v_j\notin K}\beta_j^{(j_1,0)} = \gamma_0+\sum_{v_j\notin K}\lambda_j \geq 
\gamma_0+0 \geq 0.
\end{align*}
Finally, the most interesting case is when $v_N\in K$ and $v_{j_1}\notin K$. 
In this scenario,
\begin{align*}
\sum_{v_j\notin K} \beta_j^{(j_1,0)}=\lambda_{j_1}-\gamma_0-\lambda_N+
\sum_{v_j\notin ((K\setminus\{v_N\})\cup\{v_{j_1}\})}\lambda_j \geq 0.
\end{align*}
The inequality above holds since $\lambda_{j_1}-\gamma_0\geq 0$ and $-\lambda_N\geq 0$.

Thus we have verified that \eqref{v2} holds for $\beta^{(j_1,0)}$, and
now we proceed with a similar argument using Lemma \ref{detour2} to show that \eqref{v1} is also true.
For this, again we consider four cases. 

The  cases $\{v_{j_1},v_N\}\subset K$ and
$\{v_{j_1},v_N\}\subset K^c$ are
straightforward since in these cases we have: 
\begin{align*}
\sum_{v_j\notin K} (\frac{1}{p_j}+\frac 
{\beta_j^{(j_1,0)}}k )=
\sum_{v_j\notin K} ( \frac{1}{p_j}+\frac {\lambda_j}k)
\geq m - \# K.
\end{align*}
In the third case, $v_{j_1}\in K, v_N
\notin K$, and
\begin{equation*}
\sum_{v_j\notin K} (\frac{1}{p_j}+\frac {\beta_j^{(j_1,0)}}k ) =\frac {\gamma_0}k+\sum_{v_j\notin K}
(\frac{1}{p_j}+ \frac {\lambda_j } k ) 
 > \gamma_0+m-\#K.
\end{equation*}
Finally, if $v_{j_1}\notin K$ and $v_N\in K$, 
\begin{equation*}
\sum_{v_j \notin K} (\frac{1}{p_j}+\frac {\beta_j^{(j_1,0)} }  k )
=m-\sum_{v_j\in K} (\frac{1}{p_j}+\frac {\beta_j^{(j_1,0)} } k )
 >m-\#K.
\end{equation*}
The last inequality above follows if we can prove the claim that 
\begin{align}\label{claimsm}
\frac{1}{p_j}+\frac {\beta_j^{(j_1,0)}} k <1\quad \mbox{for }\, v_j\in K.
\end{align}

To see that \eqref{claimsm} holds, consider two cases: $v_j\in K\setminus\{v_N\}$, 
and $v_j = v_N$. In the first case,
\begin{align*}
\frac{1}{p_j}+\frac {\beta_j^{(j_1,0)}} k  &\leq\frac{1}{p_j }+  \frac {\lambda_j} k  \nonumber \\
& =\sum_{i=1}^N(\frac{1}{p_i}+\frac {\lambda_i }k )-\sum_{i\neq j}(\frac{1}{p_i}+\frac { \lambda_i } k  )
\\
&<m-(m-1) = 1.
\end{align*}
%% Changed direction of sign to correct typo/error.
%
The last line above follows from \eqref{Scaling} and \eqref{v1}.
In the second case, when $j=N$, 
\begin{align*}
\frac{1}{p_N}+\frac {\beta_N^{(j_1,0)}} k\leq \frac{1}{p_N}<1,
\end{align*}
since $\beta_N^{(j_1,0)}\leq 0$ and we assume that ${1}/{p_N}<1$.

Having verified \eqref{v1} for each $\beta^{(j_1,0)}\in \zeta_1$, we move on to
the task of establishing the inequality \eqref{estimate2}. Using our assumption that the set 
$E=\{v_1, \dots, v_N\}$ is generic, we have that 
$\{v_1,\dots,v_m\}$ is a basis for $\reals^m$. Therefore we can express 
$v_N=\sum_{j=1}^m \alpha_jv_j$. Since $ \gamma _ 0 > 0$, we have
\[
|x\cdot v_N|^{\gamma_0} \leq C\sum_{j=1}^m|x\cdot v_j|^{\gamma_0}.
\]
Inserting this inequality into the form gives 
\begin{equation}\label{Spread}
\begin{aligned}
\Lambda(f_1,\dots,f_N) & = \Lambda(f_1,\dots,\frac{|x\cdot v_N|^{\gamma_0}}
{|x\cdot v_N|^{\gamma_0}}f_N)  \\
&\leq C\sum_{j=1}^m \Lambda(f_1,\dots,|x\cdot v_j|^{\gamma_0}f_j,\dots,|x\cdot v_N|^{-\gamma_0}f_N)\\
& = C\sum_{j=1}^m \Lambda(|x\cdot v_1|^{\lambda_1-\beta_1^{(j,0)}}f_1,\dots,
|x\cdot v_N|^{\lambda_N-\beta_N^{(j ,0)}}f_N).
\end{aligned}
\end{equation}
This is the estimate  \eqref{estimate2}.

We should also note that for each $\beta^{(j_1,0)}\in \zeta_1$, we have 
$\beta_N^{(j_1,0)}=\min(0,\lambda_N+\lambda_m)$. If $\beta_N^{(j_1,0)}=0$, we are done.
Otherwise, we repeat the construction as follows: For each $\beta^{(j_1,0)}\in \zeta_1$,
define $\beta^{(j_1,j_2,0)}$, with $j_2\in \{1,\dots,m\}\setminus \{j_1\}$ by
setting 
\begin{align*}
\beta^{(j_1,j_2,0)}=\beta^{(j_1,0)}-\gamma_1 e_{j_2}+\gamma_1 e_N,
\end{align*}
where $\gamma_1 = \min(\lambda_{m+1},|\lambda_N + \gamma _0|)$. Arguing as above,
we have 
$(\frac{1}{p_j},\beta^{(j_1,j_2,0)})$ satisfies \eqref{Scaling}, \eqref{v1}, 
\eqref{v2}, and \eqref{I}. Moreover,
\[
\beta_N^{(j_1,j_2,0)}=\min(0,\lambda_m+\lambda_{m+1}+\lambda_N).
\]
Set $\zeta_2=\big\{\beta^{(j_1,j_2,0)}:j_1 \in\{1,\dots,m\},j_2\in\{1,\dots,m+1\}\setminus\{j_1\}\big\}$. Continuing in this 
manner, \eqref{step} guarantees that we have $\beta_N^{(\alpha)}=0$ for all $\beta^{(\alpha)} \in \zeta_i$ for
some $i\leq \ell-m+1$.  
This completes the proof of the Lemma \ref{InductionStep}.
\end{proof}

\section{Necessary Conditions}\label{SectionNec}
%% condition sum 1/p_j \geq 1
In this section we prove Theorem \ref{necthm}. It is worth noting that the results in this section do not require the condition that the set of vectors $E$ be generic. We begin with a simple, technical Lemma.

%%% Question. Can we show that 1 /p_j + \lambda _j /k < 1. 
%% This might help simplify the proofs of necessity. Under these conditions, 
%% we would have that the characteristic function of a ball lies in the weighted space. 

\begin{lemma}\label{sect3-lemma1}
    If $V$ and $W$ are vector subspaces of $ \reals ^m$  with $V\subset W $ and 
$\{ v_1,\dots,v_\ell \} \subset  W\setminus V$, 
then there exists $w\in W\cap V^\perp$ 
    so that $w\cdot v_j \neq 0, j=1,\dots,\ell$.
\end{lemma}

Before giving the proof, we introduce some additional notation.  In the
argument below, it will be useful to consider $ \reals ^{mk}$ as a 
tensor product, $ \reals ^ m \otimes \reals ^k$. Thus, if $ x = ( x^1, 
\dots, x^m) \in \reals ^ { mk}$ with each $ x^ j \in \reals ^k$, we can 
write $ x = \sum _ { j = 1} ^ m e_j \otimes x^j$ with $ e_j$ denoting 
the unit vector in the direction of the $j$th coordinate axis.  From this it is clear that
$\reals ^m \otimes \reals ^k$ is spanned by elements of the form $ y \otimes z$ with 
$ y \in \reals ^m$ and $z \in \reals ^k$. With 
this notation, our map $ v \cdot x $  can be defined on products by $ v\cdot(y\otimes z)= (v\cdot y) z$ 
and then we use linearity to extend the definition to all of $ \reals ^m \otimes \reals ^k$. 
\begin{proof}
We inductively define $w_\alpha$ for $\alpha=1,\dots,\ell$ so that 
\begin{align*}
    w_\alpha\cdot v_i \neq 0 ,\quad i=i,\dots,\alpha \,\,\, \mbox{and } \,w_\alpha\in W \cap V^\perp.
\end{align*}
At the conclusion, we will set $w=w_\ell$.

To begin, let $w_1 = v_1^\perp$, the projection of $v_1$ onto $V^\perp$. 
Next, given $w_\alpha$, if $w_\alpha \cdot v_{\alpha+1}\neq 0$ then set $w_{\alpha+1}=w_\alpha$. 
If, on the other hand, $w_\alpha \cdot v_{\alpha+1} = 0$,  set 
$w_{\alpha+1}=cw_\alpha+v_{\alpha+1}^\perp$. Here, we choose $c$ so that 
$|v_j\cdot w_{\alpha+1}| \geq 1$, $j=1,\dots,\alpha$. For example, one can choose
\begin{align*}
    c=\frac{1+\max\{|v_{\alpha+1}^\perp\cdot v_j| : j=1,\dots,\alpha\}}
    {\min\{|w_\alpha \cdot v_j| : j=1,\dots,\alpha\}}.
\end{align*}
This completes the construction.
\end{proof}

\begin{proof}[Proof of Theorem \ref{necthm}]
\comment{
We begin by applying Lemma \ref{sect3-lemma1} to the pair of subspaces 
$\{0\} \subset \reals ^m $ to find $w_1\in \{0\}^ \perp = \reals ^m$ and  satisfying 
    Let $V\subset \reals^m$ be a subspace. We may apply Lemma \ref{sect3-lemma1}
    with the subspace $\{0\}\subset V$ to find $w\in 0^\perp\cap V = V$ such 
    that $w\cdot v_j \neq 0$ for $v_j \in V$. 
    
    When $V=\reals^m$, Lemma \ref{sect3-lemma1} gives $w_1$ so that 
}
    We begin by applying Lemma \ref{sect3-lemma1} to the pair of subspaces 
    $\{0\} \subset \reals ^m $ to find $w_1\in \{0\}^ \perp = \reals ^m$ and  
    satisfying 
    \begin{equation*}
        w_1\cdot v_j\neq 0, \qquad j=1,\dots, N.
    \end{equation*}

    Next, for $w\in \reals^m$, $u \in \reals^k$, we define 
    $w\otimes u=(w^1u, w^2 u, \dots, w^mu) \in \reals^{mk}$. Assume $|u|=1$ and then set 
\begin{align*}
    S_R = B^{mk}(Rw_1 \otimes u , \epsilon R),
\end{align*}
where the superscript on $B$ is included to make clear that we have a ball in $\reals^{mk}$.
We can choose $\epsilon>0$ small so that 
\begin{align*}
    c_1R\leq |v_j\cdot x| \leq c_2 R, \qquad x \in S_R, \ j = 1, \dots, N.
\end{align*}
This follows since 
\begin{align*}
    |v_j \cdot Rw_1 \otimes u - v_j \cdot x| \leq  |v_j| \epsilon R
\end{align*}
and 
\begin{align*}
    |v_j \cdot Rw_1 \otimes u | = R|v_j\cdot w_1| |u| = R|v_j \cdot w_1|.
\end{align*}

If we let 
\begin{align*}
    f_j(y) = \chi_{[c_1, c_2]}\left(\frac{|y|}{R}\right), \quad y\in \reals^k,
\end{align*}
then $f_j(v_j \cdot x) = 1$ if $x\in S_R$ and $\|f_j\|_{L^{p_j}_{\lambda_j}} \approx R^{k/p_j +\lambda_j}$.
Thus if \eqref{mainest} holds with the indices $(1/p_j, \lambda_j)_{j=1}^N$,  we have
$$
    cR^{mk}\leq \int_{\reals^{mk}} \prod_{j=1}^N f_j(v_j\cdot x)\, dx 
    \leq C\prod_{j=1}^N R^{k/p_j + \lambda_j}.
$$
Since this holds for all $R \in (0, \infty)$, we can conclude
\begin{align*}
    m=\sum_{j=1}^N \left(\frac{1}{p_j}+\frac{\lambda_j}{k}\right),
\end{align*}
and so \eqref{Scaling} %%(formerly \eqref{sect3-scaling}) 
is proved.

We turn to the proof of \eqref{v1c}. To this end, let $ V \subset \reals ^m$ be a 
subspace. To begin we apply Lemma \ref{sect3-lemma1} to the pair of subspaces $\{0 \} 
\subset V$ to find $w_1\in V$ 
for which $$w_1 \cdot v_j \neq 0, \quad v_j \in V.$$
If $\{v_j : v_j \notin V\}$ is a nonempty set, then we find $w_2\in V^\perp$ so that 
\begin{align*}
    w_2\cdot v_j \neq 0,\quad v_j \notin V.
\end{align*}
This time around, we set
\begin{align*}
    S_R=B^{mk}(w_1\otimes u,\epsilon)\cap (V\otimes \reals^k) + B^{mk} 
    (Rw_2\otimes u, \epsilon R)\cap (V^\perp\otimes \reals^k).
\end{align*}
Here, $V\otimes \reals^k=\{x : x= v\otimes y, v \in V, y \in \reals ^k\}$ 
and $u \in \reals ^k$ is a unit vector. Then we have 
\begin{align*}
|S_R| \approx R^{k\dim(V^\perp)}.
\end{align*}
If we choose $\epsilon >0$ small and $R_0$ large, we can find $c_1, c_2$ so that 
\begin{align*}
    c_1 \leq |v_j \cdot x| \leq c_2, \quad v_j \in V, x\in S_R,
\end{align*}
and
\begin{align*}
    Rc_1 \leq |v_j\cdot  x| \leq c_2 R, \quad v_j \notin V, \, R> R_0, \, x\in S_R.
\end{align*}
Thus, with $f_j$ defined as 
\begin{align*}
    f_j(y) = \begin{cases}  \chi_{[c_1, c_2]} (|y|), &  v_j \in V,\\[10pt]
\chi_{[c_1,c_2]} (|y|/R), & v_j \notin V,
\end{cases}
\end{align*}
%
% \begin{align*}
%     f_j(y) & =  \chi_{[c_1, c_2]} (|y|), \quad  v_j \in V,\\[10pt]
%     f_j(y) & = \chi_{[c_1,c_2]} (|y|/R), \quad v_j \notin V,
% \end{align*}
%
then it follows that $f_j(v_j\cdot x) = 1$ if $x\in S_R$ and $R>R_0$. In total, the boundedness of the form implies
\begin{align*}
    R^{k\dim(V^\perp)} & \leq C \int_{S_R} \prod_{j=1}^N f_j (v_j\cdot x) \, dx \\
    & \leq C \prod_{j=1}^N \|f_j\|_{L^{p_j}_{\lambda_j}}\\
    & \leq C\prod_{v_j\notin V} R^{k/p_j + \lambda_j}.
\end{align*}
Since this inequality holds for all $R>R_0$, we can conclude that 
\begin{align*}
    \sum_{v_j\notin V} \frac{k}{p_j} + \lambda_j \geq \dim (V^\perp).
\end{align*}

Finally, we establish \eqref{v2}, 
%%(formerly \eqref{sect3-lambdaineq}), 
the inequality for the $\lambda_j$'s. 
         Again, fix a subspace $V\subset \reals^m$. Choose $w_1, w_2$ as 
         before:  
                $w_1\in V$ so that $w_1\cdot v_j \neq 0$ for all $v_j\in V$, 
                and $w_2\in V^\perp$ 
      so that $w_2\cdot v_j \neq 0$ for all $v_j \notin 
        V$. Set 
\begin{align*}
    w_N= (w_1 + Nw_2 )\otimes u,
\end{align*}
and then define
\begin{align*}
    S_N = B^{mk}(w_N,\epsilon).
\end{align*}
Here, $u\in \reals^k$ is a unit vector as usual. For $\epsilon >0$ small, we have 
\begin{align*}
    c_1 \leq |v_jx| \leq c_2, \quad x\in S_N, \, v_j \in V,
\end{align*}
and if $N$ is large, say $N>N_0$, 
\begin{align*}
|v_j \cdot x- v_j \cdot w_N \otimes u| \leq c\epsilon, \quad x\in S_N.
\end{align*}
This time around, set 
\begin{align*}
    f_j(y) =\begin{cases}  \chi_{[c_1, c_2]} (|y|), &  v_j \in V,\\[10pt]
\chi_{B(v_j\cdot w_N u , c\epsilon)}(|y|), & v_j \notin V.
\end{cases}
\end{align*}
%
% \begin{align*}
%     f_j(y) & =  \chi_{[c_1, c_2]} (|y|), \quad  v_j \in V,\\[10pt]
%     f_j(y) & = \chi_{B(v_j\cdot w_N u , c\epsilon)}(|y|), \quad v_j \notin V,
% \end{align*}
%
Then it follows that 
\begin{align*}
    \prod_{j=1}^N f_j (v_j \cdot x) = 1, \quad \mbox{if }\, x\in S_N, \, N>N_0,
\end{align*}
and $\|f_j\|_{L^{p_j}_{\lambda_j}}\approx 1$ if $v_j\in V$ and 
$\|f_j\|_{L^{p_j}_{\lambda_j}}\approx N^{\lambda_j}$ if $v_j \notin V$.

Thus for $N \geq N_0$,  the boundedness of the form implies
$$
 C \leq \int_{\reals^{mk}} \prod_{j=1}^N f_j (v_j \cdot x)\,dx  
 \leq \prod_{j=1}^N \| f_k\|_{L^{p_j}_{\lambda_j}} \\
 \approx \prod_{v_j\notin V} N^{\lambda_j}.
$$
Since this inequality holds for all $N>N_0$, we have $\sum_{v_j\notin V} \lambda_j \geq 0$ as desired.

Our last step is to establish that the  condition \eqref{I} must hold if we have \eqref{mainest}. 
Towards this end, we define 
$$
f_j(t) = \sum _ { \ell =1 } ^ \infty a_{j,\ell} 2 ^ { -\ell (k /p_j + \lambda _j)} 
\chi_ { [2^\ell ,2^{\ell +1} ) }(|t|), \quad t \in \reals ^k, 
$$
where we assume each $ a_{j,\ell} \geq 0$. With this choice, we have 
$$\| f_j \| _ { L^{p_j}_{\lambda _j } } \leq C \left (\sum_{\ell = 1 } ^ \infty
a_ { j, \ell } ^ { p_j } \right)^{1/p_j}.$$
We will set $ a_ { j, \ell} = \ell ^ { -(1+\epsilon)/p_j}$ 
where $\epsilon >0 $. With this choice, we have 
$f_ j \in L ^{ p_j }_{ \lambda _j }.$

Next, we use Lemma \ref{sect3-lemma1} to find $w \in \reals ^m$ so that 
\begin{equation} \label{notzero}
v_j \cdot w \neq 0, \quad v_j \in E. 
\end{equation}
As before we choose a unit vector 
$u \in \reals ^k$ and set $ S_\ell = B^ { mk}(2^ \ell w \otimes u, 2^ \ell \epsilon)$. 
Thanks to \eqref{notzero} and our choice of $ a_{j,\ell}$, we may find  $\ell _0$
so that 
\begin{equation}
    \label{lowerbound}
f_j ( v_j \cdot x )\geq c \ell ^ { - (1+ \epsilon)/p_j } 2 ^ { -\ell ( k/p_j + \lambda _j ) }
, \qquad \ell \geq \ell _0, \,x \in S_\ell.
\end{equation}
Using \eqref{lowerbound} and that $ |S_ \ell | \approx 2 ^ {k\ell m}$, we have
$$
\sum _ { \ell \geq \ell _0 } 2^ { klm}\prod _ { j =1 }^N\ell^ { -( 1+ \epsilon)/p_j}2^{ -\ell ( k/p_j + \lambda _j ) }
\leq \sum _ { \ell = \ell _0} ^ \infty \int _ { S_ \ell} \prod _ { j=1 } ^N f_j ( v_j \cdot x ) \, dx 
\leq C \prod _ { j =1 }^N \| f_j \| _ { L^ { p_j }_ { \lambda _j}},
$$
where the last inequality follows since we assume the estimate \eqref{mainest}.  
Using our observation about $\| f_j \| _{ L^ {p_j}_{ \lambda _j}}$, we may conclude
$$ \sum _{ \ell = \ell _0} ^ \infty \ell ^ { - (1+ \epsilon) \sum _ { j =1 } ^N 1/p_j} 
\leq C (\sum _ { \ell =1 } ^ \infty \ell ^ { -(1+ \epsilon)} ) ^ { \sum _ { j =1} ^ N 1/p_j}.$$
In particular, we have that the sum $ \sum _ { \ell = \ell _0} ^ \infty 
\ell ^ {-(1+ \epsilon)\sum _ { j =1} ^N 1/p_j}$ is finite 
for each $ \epsilon >0$. This implies that we must have \eqref{I}. 
\end{proof}

\comment{See, in order: form.pdf Theorem 1.7, Lemma 1.16, and simple-forms.pdf 
Theorem 1.13. but need to review earlier proofs to make sure that the results apply to this
context. Need to review earlier proofs to make sure results are
proved in this case.}

\section{Examples}\label{examples}

In this section, we give several of examples of forms for where we can 
    use our theorem to study the boundedness.   We also give an 
        example of a non-generic set of vectors where a naive extension of 
                our methods fails to establish boundedness 
                in the interior of the polytope defined 
                    by the necessary conditions in 
                       Theorem \ref{necthm}.

\begin{example}
Let $E= \{ e_1+\dots + e_m, e_1, \dots, e_m  \} \subset \reals ^m$. Then it is 
easy to see that the set $E$ is generic. Thus Theorem \ref{suffthm} will apply to the form 
$$
\int _ {\reals ^m} f_0(x_1+ \dots + x_m)\prod _{ i =1 } ^m f_i(x_i)  \, dx. 
$$
\end{example}

Our next example shows how to generate arbitrarily large sets of generic  
vectors $E$ and helps to justify our use of the term generic.  
\begin{example} 
If $E_N= \{ v_1, \dots, v_N\} $ is a generic family in $ \reals ^m$, we may add an 
additional vector $v_{N+1} $ so that $ E_{N+1} = \{ v_1, \dots, v_{ N+1}\}$ is a 
generic set. By the generic condition on $E_N$, for each subset 
$K \subset E_N$ with $\#K = m-1$, 
$\spn (K)$ is a subspace of dimension $m-1$. As $ \reals ^m$ 
cannot be the union of a finite number of proper subspaces, we may 
find a vector $v_{N+1}$ which does not belong to any of the subspaces 
spanned the subsets of $E_N$ of cardinality $m-1$. Thus, 
the set $E_{N+1}$ is generic. 
For completeness, we note that we can begin with the standard 
basis and then use the inductive step above to 
generate large generic sets of vectors. 
\end{example} 
\smallskip

Our final observation shows some of the aforementioned limitations of the current methods. As we note in 
Theorem \ref{necthm}, our necessary conditions do  not require the set of vectors
$E$ to be generic. When we began this work, we had 
hoped to  establish 
a result that was close to necessary and sufficient for forms that are based on 
general sets of vectors $E$. Unfortunately, we are not able to do this.  
We will give a non-generic set of five vectors in $ \reals ^3$ and show that a naive attempt to extend the proof
of Theorem \ref{suffthm} fails to establish that \eqref{mainest} holds for
indices in the interior of the polytope given by Theorem \ref{necthm}.

\begin{example} 
Consider the set of 5 vectors in $ \reals ^3$ given by
$$ E = \{ e_1, e_1 + e_2, e_1+ e_3, e_1-e_2, e_1-e_3\}. $$
We will label these vectors by  $v_1=e_1, v_2= e_1+ e_2, v_3 =
e_ 1+ e_3, v_4 = e_1 - e_2, v_5 = e_1-e_3$.
Note that there are two dependent sets of three vectors in $E$.

As before, the result of Bennett {\em et~al.}~characterizes the 
indices for which we have \eqref{mainest}   when the exponents for 
the weights $\lambda _j $ are zero.  
Continuing, we may extend the result to Lorentz spaces and use 
the generalized H\"older's inequality to obtain \eqref{mainest} in 
the case that the weights are non-negative. We only need to note 
that as long as no two vectors are collinear, we may use 
\eqref{v1} as in the proof of \eqref{NotOne} to conclude that 
$ 0 < 1/p_j + \lambda _j  <1$ (here $k=1$). 

We consider the vector  of indices
$$
(1/p_1,  \dots, 1/p_5,\lambda_1,\dots, \lambda _5)
=(11/15,  6/15,  2/3, 6/15,  2/3,  -2/15, 2/15, 0, 2/15, 0). 
$$
We leave it as an exercise to verify that this vector of 
indices satisfies the conditions of Theorem 
\ref{suffthm}. The most interesting condition to check is \eqref{v1}, and here one must note that because of the linear dependencies $2v_1 = v_2 + v_4$ and $2v_1 = v_3 + v_5$, there are two sums in \eqref{v1} to check where $\mbox{dim}(V)=2$ and the condition $v_j\notin V$ is true for two indices $j$ rather than three due to the linear dependence.

We observe that we have a linear dependency 
$ 2v_1 = v_2 + v _ 4$ and  arguing as in \eqref{Spread}
we can show that
$$\Lambda ( f_1, \dots f_5)
\leq C ( \Lambda ( |\cdot |^ { -2/15}f_1, |\cdot|^{2/15}f_2,
f_3, f_4, f_5 ) 
+  \Lambda ( |\cdot |^ { -2/15}f_1, f_2,
f_3,|\cdot|^{2/15} f_4, f_5 )). 
$$
We would like to have the estimate \eqref{mainest} 
for the vectors of indices:
\begin{align*}
(1/p_1, \dots, 1/p_5, \beta_1^{(2,0)}, \dots \beta_5 ^{(2,0)}) &=
(11/15, 6/15,  2/3,  6/15,  2/3, 0 , 0, 0,2/15,0 ),\\
(1/p_1,  \dots, 1/p_5,\beta_1 ^ { (4,0)}, \dots, \beta_5 ^{(4,0)}) 
&=(11/15,  6/15,  2/3,  6/15, 2/3, 0, 2/15, 0,0,0 ).
\end{align*}
If we consider the inequality \eqref{v1} for the subspace 
$V = \spn \{ v_1, v_3, v_5\} = \spn \{ e_1, e_3\}$, we see that 
for both $\beta^{(2,0)} $ and $ \beta^{(4,0)}$
$$
\sum _{ v_j \notin V } \frac 1 { p_j } + \beta ^{(k,0)}_j = \frac 1 { p_2 } 
+ \beta^{(k,0)} _2 + \frac 1 { p_4} + \beta^{(k,0)} _4
= \frac { 14}{ 15} <1=  3 - \dim V. 
$$
% 
% If we consider the inequality \eqref{v1} for the subspace 
% $V = \spn \{ v_1, v_2, v_4\} = \spn \{ e_1, e_2\}$, we see that 
% for both $\beta^{(2,0)} $ and $ \beta^{(4,0)}$
% $$
% \sum _{ v_j \notin V } \frac 1 { p_j } + \beta ^{(k,0)}_j = \frac 1 { p_3 } 
% + \beta^{(k,0)} _3 + \frac 1 { p_5} + \beta^{(k,0)} _5
% = \frac { 14}{ 15} <1=  3 - \dim V. 
% $$
%
Thus, this avenue of proof would require us to use an estimate which Theorem \ref{necthm}
tells us cannot hold. 
\end{example} 

We close by listing several problems related to our work worthy of further study. 

\begin{itemize}
    \item Find conditions which are close to necessary and sufficient for forms that
    are based on general sets of vectors $E$, rather than being restricted to 
    generic sets of vectors. 
    \item Consider weighted estimates for the more general forms such as those studied in 
    Bennett {\em et~al.} \cite{MR2661170}. 
    \item Find conditions on families of more general weights which 
allow us to establish that a Brascamp-Lieb form is bounded on weighted $L^p$-spaces. 
\end{itemize}
%% What goes wrong? 

%% Section of the general case moved to scraps. 

%%%%%%%%%%%%%%%%%%%%%%%%%%%%%%%%%%%%%
%Bibliography
%%%%%%%%%%%%%%%%%%%%%%%%%%%%%%%%%%%%%
\newpage
\bibliographystyle{plain}

\bibliography{main,inverse}

\def\cprime{$'$} \def\cprime{$'$} \def\cprime{$'$} \def\cprime{$'$}
  \def\cprime{$'$} \def\cprime{$'$} \def\cprime{$'$} \def\cprime{$'$}
  \def\cprime{$'$} \def\cprime{$'$}
\begin{thebibliography}{10}

\bibitem{MR1650312}
Franck Barthe.
\newblock On a reverse form of the {B}rascamp-{L}ieb inequality.
\newblock {\em Invent. Math.}, 134(2):335--361, 1998.

\bibitem{BC:1988}
R.~Beals and R.R. Coifman.
\newblock The spectral problem for the {D}avey-{S}tewartson and {I}shimori
  hierarchies.
\newblock In {\em Nonlinear evolution equations: Integrability and spectral
  methods}, pages 15--23. Manchester University Press, 1988.

\bibitem{MR2661170}
Jonathan Bennett, Anthony Carbery, Michael Christ, and Terence Tao.
\newblock Finite bounds for {H}\"older-{B}rascamp-{L}ieb multilinear
  inequalities.
\newblock {\em Math. Res. Lett.}, 17(4):647--666, 2010.

\bibitem{MR3665018}
Neal Bez, Sanghyuk Lee, Shohei Nakamura, and Yoshihiro Sawano.
\newblock Sharpness of the {B}rascamp-{L}ieb inequality in {L}orentz spaces.
\newblock {\em Electron. Res. Announc. Math. Sci.}, 24:53--63, 2017.

\bibitem{MR0346109}
H.~J. Brascamp, Elliott~H. Lieb, and J.~M. Luttinger.
\newblock A general rearrangement inequality for multiple integrals.
\newblock {\em J. Functional Analysis}, 17:227--237, 1974.

\bibitem{MR1871279}
R.~M. Brown.
\newblock Estimates for the scattering map associated with a two-dimensional
  first-order system.
\newblock {\em J. Nonlinear Sci.}, 11(6):459--471, 2001.

\bibitem{MR4198080}
R.~M. Brown, C.~W. Lee, and K.~A. Ott.
\newblock Estimates for {B}rascamp-{L}ieb forms in {$L^p$}-spaces with power
  weights.
\newblock {\em Proc. Amer. Math. Soc.}, 149(2):747--760, 2021.

\bibitem{MR3494243}
R.~M. Brown, K.~A. Ott, and P.~A. Perry.
\newblock Action of a scattering map on weighted {S}obolev spaces in the plane.
\newblock {\em J. Funct. Anal.}, 271(1):85--106, 2016.

\bibitem{MR2077162}
E.~A. Carlen, E.~H. Lieb, and M.~Loss.
\newblock A sharp analog of {Y}oung's inequality on {$S^N$} and related entropy
  inequalities.
\newblock {\em J. Geom. Anal.}, 14(3):487--520, 2004.

\bibitem{MR766216}
M.~Christ.
\newblock On the restriction of the {F}ourier transform to curves: endpoint
  results and the degenerate case.
\newblock {\em Trans. Amer. Math. Soc.}, 287(1):223--238, 1985.

\bibitem{FA:1984}
A.S. Fokas and M.J. Ablowitz.
\newblock On the inverse scattering transform of multidimensional nonlinear
  equations related to first-order systems in the plane.
\newblock {\em J. Math. Phys.}, 25:2494--2505, 1984.

\bibitem{MR1164632}
Loukas Grafakos.
\newblock On multilinear fractional integrals.
\newblock {\em Studia Math.}, 102(1):49--56, 1992.

\bibitem{MR1812822}
Loukas Grafakos and Nigel Kalton.
\newblock Some remarks on multilinear maps and interpolation.
\newblock {\em Math. Ann.}, 319(1):151--180, 2001.

\bibitem{SJ:1986}
S.~Janson.
\newblock On interpolation of multi-linear operators.
\newblock {\em Function Spaces and Applications (Proceedings, Lund 1986),
  Lecture Notes in Math}, 1302:290--302, 1986.

\bibitem{MR1682725}
Carlos~E. Kenig and Elias~M. Stein.
\newblock Multilinear estimates and fractional integration.
\newblock {\em Math. Res. Lett.}, 6(1):1--15, 1999.

\bibitem{MR4047697}
Yasuo Komori-Furuya.
\newblock Weighted estimates for bilinear fractional integral operators: a
  necessary and sufficient condition for power weights.
\newblock {\em Collect. Math.}, 71(1):25--37, 2020.

\bibitem{MR3130311}
Kabe Moen.
\newblock New weighted estimates for bilinear fractional integral operators.
\newblock {\em Trans. Amer. Math. Soc.}, 366(2):627--646, 2014.

\bibitem{MR4081134}
Adrian Nachman, Idan Regev, and Daniel Tataru.
\newblock A nonlinear {P}lancherel theorem with applications to global
  well-posedness for the defocusing {D}avey-{S}tewartson equation and to the
  inverse boundary value problem of {C}alder\'{o}n.
\newblock {\em Invent. Math.}, 220(2):395--451, 2020.

\bibitem{ZN:2009}
Z.~Nie.
\newblock {\em Estimates for a class of multi-linear forms}.
\newblock PhD thesis, University of Kentucky, 2009.

\bibitem{MR2754810}
Zhongyi Nie and Russell~M. Brown.
\newblock Estimates for a family of multi-linear forms.
\newblock {\em J. Math. Anal. Appl.}, 377(1):79--87, 2011.

\bibitem{MR0146673}
Richard O'Neil.
\newblock Convolution operators and {$L(p,\,q)$} spaces.
\newblock {\em Duke Math. J.}, 30:129--142, 1963.

\bibitem{MR3551174}
Peter~A. Perry.
\newblock Global well-posedness and long-time asymptotics for the defocussing
  {D}avey-{S}tewartson {II} equation in {$H^{1,1}(\mathbb{C})$}.
\newblock {\em J. Spectr. Theory}, 6(3):429--481, 2016.
\newblock With an appendix by Michael Christ.

\bibitem{MR0098285}
E.~M. Stein and Guido Weiss.
\newblock Fractional integrals on {$n$}-dimensional {E}uclidean space.
\newblock {\em J. Math. Mech.}, 7:503--514, 1958.

\bibitem{SW:1971}
E.M. Stein and Guido Weiss.
\newblock {\em Introduction to {F}ourier analysis on {E}uclidean spaces}.
\newblock Princeton University Press, 1971.

\end{thebibliography}

\smallskip
\noindent
\today

\end{document}